\documentclass[11pt]{amsart}
\usepackage{amsmath}

\newcommand{\abs}[1]{\vert #1\vert}  
\newcommand{\ov}{\overline}
\newcommand{\ii}{\mathrm{i}}
\newcommand{\dif}{\mathrm{d}}

\DeclareMathOperator{\tr}{trace}
\DeclareMathOperator{\di}{div}
\DeclareMathOperator{\gr}{grad}

\DeclareMathOperator{\Ric}{Ric}
\DeclareMathOperator{\Hess}{Hess}

\newcommand{\CC}{\mathbb{C}}

\newcommand{\DD}{\mathcal{D}} 
\newcommand{\RR}{\mathbb{R}}

\newtheorem{te}{Theorem}
\newtheorem{pr}{Proposition}
\newtheorem{co}{Corollary}
\newtheorem{lm}{Lemma}
\newtheorem{de}{Definition}    
\newtheorem{re}{Remark}
\newtheorem{ex}{Example}

\begin{document}

\title[Biharmonic holomorphic maps]{Biharmonic holomorphic maps and conformally K\"ahler geometry} 

\author{M. Benyounes, E. Loubeau}
\address{D\'epartement de Math\'ematiques, Universit\'e de Bretagne Occidentale, 6,
avenue Victor Le Gorgeu, CS 93837, 29238 Brest Cedex 3, France}
\email{Michele.Benyounes@univ-brest.fr $\text{and}$ Eric.Loubeau@univ-brest.fr}

\author{R. Slobodeanu}
\address{Department of Theoretical Physics and Mathematics, Faculty of Physics, University of Bucharest, P.O. Box Mg-11, RO--077125 Bucharest-M\u agurele, Romania.}
\email{radualexandru.slobodeanu@g.unibuc.ro}

\thanks{The authors thank to Liviu Ornea for many helpful comments on l.c.K. geometry. This article was written while the third author was invited professor at the University of Brest.}

\subjclass[2010]{32Q60, 58E20.}

\keywords{Biharmonic map, holomorphic map, Lee vector field, locally conformally K\"ahler manifold.}

\begin{abstract}
We give conditions on the Lee vector field of an almost Hermitian manifold such that any holomorphic map from this manifold into a  $(1,2)$-symplectic manifold must satisfy the fourth-order condition of being biharmonic, hence generalizing the Lichnerowicz theorem on harmonic maps. These third-order non-linear  conditions are shown to greatly simplify on l.c.K. manifolds and construction methods and examples are given in all dimensions.
\end{abstract}

\maketitle

\section{Introduction}

Whether seen as a generalization of closed geodesics or an infinite dimensional Morse theory, harmonic maps are one of the most important subjects among geometric variational problems. Their study has given rise to a whole array of questions from existence to regularity and stability and has been the impetus to many new tools and techniques. The subtle balance of this problem, born out of its dual nature, analytical and geometrical, has led to an existence theory marked by contrast between abundance results (the original Eells-Sampson theorem or maps between two-spheres), partial existence (e.g. from the two-torus to the two-sphere) and total lack, for example by applying the Bochner formula or sections with the Sasaki metric (\cite{EL1,EL2,EL3}).

When harmonic maps are absent or found wanting, one can seek mappings as close to harmonic as possible without actually being so, by devising a functional measuring the failure of harmonicity, i.e. the $L^2$-norm of the tension field, its vanishing characterizing harmonic maps. Critical points of this bi-energy are called biharmonic maps and its associated Euler-Lagrange equation is a fourth-order system of elliptic PDE's. The high order of this equation as well as the explicit curvature term involved in it, make not only existence results difficult to reach but even examples, or families of examples, hard to come by. While some constructions exist, notably the composition of minimal immersions with the 45-th parallel map of spheres, in most instances biharmonic maps are few and far in between.

One redeeming feature is that, as a variational problem, all the techniques and tools employed to investigate harmonic maps can be tried on biharmonic maps, in particular, and this in spite of the more analytical than geometrical essence of biharmonicity, constructions guided by geometric insight to reduce the complexity of the problem. The first and foremost of such techniques is to work with equivariant maps and lower the number of variables, yielding a fourth-order ODE, which carries its own difficulties (cf.~\cite{M-R}).

An alternative objective would be to decrease the order of the (system of) equations, combining an auxiliary equation with a geometric condition, e.g. turning harmonic maps into biharmonic ones by conformal transformations (\cite{BK}). In the context of Hermitian geometry, a readily available class of mappings is the holomorphic ones and the blue-print for this work is the article of Lichnerowicz (\cite{L}) giving conditions on almost complex structures (a cosymplectic domain and $(1,2)$-symplectic target) ensuring that holomorphic maps are harmonic, hence uncoupling the second-order tension field into two first-order equations. For the fourth-order equation defining biharmonic maps, there does not appear to be any corresponding (metric-independent) second-order condition which would enable an even splitting of the order of derivation between maps and complex structures.

The purpose of this article is to follow Lichnerowicz's original approach and show that it is indeed possible to determine conditions (necessarily weaker than in~\cite{L}) such that holomorphic maps satisfy the harder condition of biharmonicity, and produce examples. These conditions are perforce of third order and non-linear, and difficult to solve to the extent of being seemingly unproductive. Surprisingly, when working with the class of locally conformal K\"ahler (l.c.K.) manifolds, they become more amenable to geometric interpretation to the point of allowing the construction of some examples.

In the Gray-Hervella classification (\cite{GH}) of almost Hermitian manifolds, (1,2)-symplectic manifolds correspond to the class $\mathcal{W}_1 \oplus \mathcal{W}_2$, while the cosymplectic ones form $\mathcal{W}_1 \oplus \mathcal{W}_2 \oplus \mathcal{W}_3$. If we assume the codomain to be (1,2)-symplectic, in order to avoid recovering Lichnerowicz's result, the most reasonable class to consider for the domain is $\mathcal{W}_4$ which contains l.c.K. manifolds. These manifolds are defined by the existence of an open cover such that locally the metric is conformal to a K\"ahler metric and, crucially to us, this translates into the existence of a global closed vector field $B$, the Lee vector field, proportional to the divergence of the almost complex structure, whose dual one-form $\theta$ satisfies the integrability condition of the fundamental two-form $\Omega$: $\dif \Omega=\theta \wedge \Omega$. 

Locally conformal K\"ahler manifolds have been extensively studied since Vaisman's pioneering paper~\cite{Vais0} and the monograph \cite{DO} and surveys \cite{orn1, orn2} give an excellent account of the latest developments. 

On the one hand, some manifolds not supporting any K\"ahler  structure, turn out to admit l.c.K. metrics. Compact examples include diagonal Hopf manifolds (\cite{GO}), Inoue surfaces (\cite{Tri}) and the complex surfaces classified by Belgun (\cite{Bel}), Oeljeklaus-Toma manifolds (\cite{OT}) and generalized Thurston manifolds, while non-compact examples are rarer, e.g. \cite{renaud}. On the other hand, on manifolds known to admit a K\"ahler structure one can take representatives in the conformal class of the original metric turning them into (non-K\"ahler) globally conformal K\"ahler (g.c.K) manifolds, or prescribe new (non-g.c.K.) l.c.K. metrics as in \cite{Vais0}. By the local nature of our equations, any example can serve as testing ground for the biharmonicity conditions of Theorem \ref{theo} and the wealth of l.c.K. geometry allows enough freedom to construct non-trivial examples on which holomorphic maps are automatically biharmonic. Note that neither of the two conditions of Theorem~\ref{theo} is automatically satisfied on l.c.K. manifolds and the final section should help establishing general conditions on almost Hermitian manifolds making holomorphic maps biharmonic. 

Throughout the paper, manifolds, metrics, and maps are assumed to be smooth, and $(M,g)$ is a connected Riemannian manifold. We denote by $\nabla$ the Levi-Civita connection of $(M,g)$, and we use the following sign conventions for the curvature tensor field
$$R(X,Y)Z=\nabla_X\nabla_Y Z-\nabla_Y \nabla_X Z-\nabla_{[X,Y]}Z,$$  and $\Delta f = \tr \nabla \dif f$ for the Laplacian on functions. 

If $(M,J)$ is an almost complex manifold, we denote its real and complex dimensions by $\dim_{\RR} M=m$ and $\dim_{\CC} M=n$, so $m=2n$. The splitting of the complexified tangent bundle in $\pm \mathrm{i}$-eigenspaces is
$T^{\CC}M = T^{\prime}M \oplus T^{\prime \prime}M$, and an adapted orthonormal frame of the form
$\{ Z_{j}=\tfrac{1}{\sqrt{2}}(e_j - \mathrm{i}Je_j), Z_{\bar{\jmath}}=\ov{Z}_j \}_{j=1, 2, ..., n}$ will be called a \emph{Hermitian frame}. Summation on repeated indices is also assumed.

\section{Biharmonic maps}

Let $\varphi:(M,g) \to (N,h)$ be a smooth map between Riemannian manifolds and let $\nabla^{\varphi}$ denote the usual connection on the pull-back bundle, $\varphi^{-1}TN$. The rough Laplacian and Ricci operator on a section $v$ of $\varphi^{-1}TN$ are 
$$\tr (\nabla^\varphi)^2 v =
\sum_{i=1}^m  \nabla_{e_i}^{\varphi}\nabla_{e_i}^{\varphi}v
- \nabla_{\nabla_{e_i}^M e_i}^{\varphi} v 
$$
$$
\Ric^{\varphi}v=\sum_{i=1}^m R^N (v, \dif \varphi(e_i))\dif \varphi(e_i),
$$
where $\{e_i\}_{i=1,...,m}$ is any local orthonormal frame on $M$. 

Recall that the \textit{tension field} of $\varphi$ is defined as
$\tau(\varphi)=\tr \nabla \dif \varphi$, where $\nabla \dif \varphi$ is the second fundamental form of $\varphi$. Then the \textit{bienergy} of the map $\varphi$ is (\cite{EL1})
$$
E_2(\varphi) =\dfrac{1}{2}\int_M \abs{\tau(\varphi)}^2 \nu_g,
$$
and $\varphi$ is \textit{biharmonic} if it is a critical point of $E_2$, equivalently, if it satisfies the associated
Euler-Lagrange equation, i.e. the vanishing of the \textit{bitension field}, (\cite{Jia}) 
\begin{equation}\label{biha}
\tau_2(\varphi)= \big(\tr (\nabla^\varphi)^2+\Ric ^{\varphi}\big)\big(\tau(\varphi)\big)=0. 
\end{equation}

Clearly harmonic maps are biharmonic but, if $M$ is compact and $N$ has non-positive sectional curvature, \textit{any} biharmonic map is harmonic (\cite{Jia}).

For isometric immersions, compactness cannot be dropped (\cite{TOu}), nevertheless biharmonic surfaces in the 3-dimensional Euclidean or hyperbolic space forms must be minimal (\cite{CMO+, Jia}) and in higher dimensions, biharmonic \textit{properly immersed} submanifolds in $\mathbb{E}^k$ are minimal (\cite{Ak}).

Contrastingly for positively curved codomains constructions are known, e.g. biharmonic curves in $\mathbb{S}^{n}$ are explicitly known (\cite{CMO+}), and the composition of a minimal immersion with the 45-th parallel immersion of 
$\mathbb{S}^{n-1}(1/\sqrt{2})$ in $\mathbb{S}^{n}$ is biharmonic but not harmonic (\cite{CMO+}). Moreover, $\mathbb{S}^{2}(1/\sqrt{2})$ is essentially the only nonharmonic biharmonic surface of $\mathbb{S}^{3}$ (\cite{CMO}), while, in $\mathbb{S}^4$ there exist closed orientable embedded nonminimal biharmonic surfaces of arbitrary genus (\cite{CMO+}). In the non-orientable case, a nonminimal biharmonic embedding of $\RR P^2$ in $\mathbb{S}^5$ is explicitly given in \cite{CMO+}.

Examples of (conformal) biharmonic maps between equidimensional manifolds are also known in dimension 3 and 4 (\cite{BFO, BK}), e.g. the inverse map of the stereographic projection of $\mathbb{S}^4$ or the inversion on $\RR^{4}\setminus \{0\}$. 

Submersive biharmonic maps have also been constructed from the Hopf map, after composition with an inversion or stereographic projection (\cite{BFO}).  
A more general scheme (\cite{BK, Sed}) is to start with some harmonic map and render it biharmonic by conformal transformations of the (co)domain metric; this yields the construction of  examples from $\RR^{4}$ to $\RR \times \mathbb{S}^2$, from $\RR^{3}$ to $\RR^{2}$ or the identity map of $\RR_{+}^{m}$. However the picture is more rigid in low dimensions, as a biharmonic Riemannian submersion from a 3-dimensional space form into a surface is necessarily harmonic (\cite{WOu}).

This quick review of the construction of (nonharmonic) biharmonic maps shows how sporadic the available examples are and highlights the need for a more systematic approach. Since most of the techniques developed for harmonic maps apply to biharmonic maps, though usually harder to implement, in the next section we will follow Lichnerowicz' method to obtain conditions such that holomorphic maps become biharmonic.

\section{Biharmonic holomorphic maps}

This section establishes conditions on the domain forcing holomorphic maps into a $(1,2)$-symplectic target, to be biharmonic. 
The key observation is that for such mappings the tension field is nothing but the image of the (scaled) Lee vector field and thus  the computation of the bitension field only requires a formula on the commutation of the differential of the map and second order covariant derivatives.
This then easily leads, in Theorem \ref{theo}, to the identification of a pair of sufficient conditions on the Lee vector field. Though these equations are automatically satisfied by cosymplectic structures, in general they are unwieldy and, in the following section, we will have to consider a quite particular class of complex manifolds before being able to construct solutions. Since all of the computations are carried out on the complexification of the tangent bundle, the resulting conditions can be re-written (Corollary \ref{real}) on the real tangent bundle, as required by the nature of our problem.

\begin{pr} \label{chain}
Let $\varphi: (M,g) \to (N,h)$ be a smooth map between Riemannian manifolds. If $V$ is a smooth vector field on $M$, then
\begin{equation}\label{ohn}
\begin{split}
\big(\tr (\nabla^\varphi)^2+\Ric ^{\varphi}\big)\big(\dif \varphi(V)\big) = & \, \dif \varphi\big((\tr \nabla^2 + \Ric)(V)\big) \\
& + \nabla_{V}^{\varphi}\tau(\varphi) + 2\mathrm{trace} \nabla \dif \varphi(\cdot, \nabla_{\cdot}V).
\end{split}
\end{equation}
\end{pr}

\begin{proof}
Consider two Riemannian manifolds $(M^m,g)$ and $(N,h)$, a vector field $V$ on $M$ and a smooth map $\varphi: M \to N$. Let $\{e_i\}_{i=1,..., m}$ be a local orthonormal frame on $(M,g)$. Then
\begin{align*}
\nabla^{\varphi}_{e_i}\nabla^{\varphi}_{e_i}\dif \varphi(V)&= \nabla^{\varphi}_{e_i}\big(\nabla \dif \varphi (e_i, V) + \dif \varphi(\nabla_{e_i} V)\big)\\
&=\nabla^{\varphi}_{e_i}\nabla \dif \varphi(e_i, V) +\nabla \dif \varphi(e_i, \nabla_{e_i} V)+ \dif \varphi(\nabla_{e_i}\nabla_{e_i} V);\\
-\nabla^{\varphi}_{\nabla_{e_i} e_i}\dif \varphi(V)&=
-\nabla \dif \varphi(\nabla_{e_i} e_i, V) -
\dif \varphi(\nabla_{\nabla_{e_i} e_i} V)
\end{align*}
and 
\begin{align*}
\Ric^{\varphi} (\dif \varphi(V))&=
\nabla^{\varphi}_{V}\nabla^{\varphi}_{e_i}\dif \varphi(e_i)
-\nabla^{\varphi}_{e_i}\nabla^{\varphi}_{V}\dif \varphi(e_i)
-\nabla^{\varphi}_{[V, e_i]}\dif \varphi(e_i)\\
&=
\nabla^{\varphi}_{V}\tau(\varphi)+ 
\nabla \dif \varphi(V, \nabla_{e_i} e_i)+ 
\dif \varphi(\nabla_{V}\nabla_{e_i} e_i)\\
&-\nabla^{\varphi}_{e_i}\nabla \dif \varphi(V, e_i)
- \nabla \dif \varphi(e_i, \nabla_V e_i) 
- \dif \varphi(\nabla_{e_i} \nabla_{V} e_i)\\
&-\nabla \dif \varphi([V, e_i], e_i)
- \dif \varphi\left(\nabla_{[V, e_i]}e_i \right).
\end{align*}
Summing these terms yields
\begin{equation*}
\begin{split}
\big(\tr (\nabla^\varphi)^2+\Ric^{\varphi}\big)(\dif \varphi(V)) = & \, \dif \varphi\big((\tr \nabla^2 + \Ric)(V)\big) \\
& + \nabla_{V}^{\varphi}\tau(\varphi) + 2 \nabla \dif \varphi(e_i, \nabla_{e_i}V)-2 \nabla \dif \varphi(e_i, \nabla_{V}e_i),
\end{split}
\end{equation*}
but
$$
\nabla \dif \varphi(e_i, \nabla_V e_i)= \sum_{1\leq i<j \leq m}\big(g(\nabla_V e_j, e_i ) + g(\nabla_V e_i, e_j)\big) \nabla \dif \varphi(e_i, e_j)=0,
$$
by symmetry of the second fundamental form, and the result follows.
\end{proof}

\begin{te}\label{theo}
Let $(M, J, g)$ be an almost Hermitian manifold such that the vector field $\sigma = J\di J$ satisfies 
\begin{equation}\label{bihol}
g (\nabla_{Z}\sigma, W) + g (\nabla_{W}\sigma, Z) = g (\sigma, Z) g(\sigma, W),
\end{equation}
for all $Z$ and $W$ in $T^{\prime}M$, and
\begin{equation}\label{bihol'}
\left(\tr \nabla^2 \sigma + \Ric \sigma -\nabla_{\sigma}\sigma
- 2\left(\nabla_{Z_j}(\nabla_{\ov Z_j}\sigma)^{\prime \prime}
+\nabla_{(\nabla_{Z_j}\sigma)^{\prime}}\ov Z_j - \nabla_{\sigma^{\prime}}\sigma^{\prime \prime}\right)\right)^{\prime} = 0,
\end{equation}
where $\{ Z_{j}, Z_{\ov\jmath}\}_{j=1,2,...,n}$ is a Hermitian frame.

Then any holomorphic map from $(M, J, g)$ into a $(1,2)$-symplectic almost Hermitian manifold is biharmonic.
\end{te}

\begin{proof}
Let $(M, J, g)$ be an almost Hermitian manifold and 
$(N, J^N, h)$ a $(1,2)$-symplectic almost Hermitian manifold. Let $\varphi: (M, J) \to (N, J^N)$ be a holomorphic map, then its tension field is given by (\cite{BW}) $$\displaystyle{\tau(\varphi)=-\dif\varphi(\sigma)}.$$
Let $\{ Z_{j}, Z_{\ov\jmath} \}_{j=1,2,...,n}$ be a Hermitian frame on $T^\CC M$, then, by Equation~ \eqref{ohn}, the bitension field of $\varphi$ is
\begin{align*}
\tau_2(\varphi) = & \, \dif \varphi \left( \tr \nabla^2 \sigma + \Ric \sigma - \nabla_{\sigma}\sigma \right)\\
&-\nabla \dif \varphi(\sigma, \sigma) +
2\nabla \dif \varphi(Z_j, \nabla_{Z_{\ov \jmath}}\sigma)+2\nabla \dif \varphi(Z_{\ov \jmath}, \nabla_{Z_j}\sigma).
\end{align*}
As the bitension field is a real operator, it is enough to consider the vanishing of its $\varphi^{-1}T^\prime N$-component:
\begin{align*}
[\tau_2(\varphi)]^\prime = & \, \dif \varphi \left( \tr \nabla^2 \sigma + \Ric \sigma - \nabla_{\sigma}\sigma \right)^\prime\\
&+\left(2g(\nabla_{Z_{\ov{\jmath}}}\sigma, Z_{\ov{k}})
- g(\sigma, Z_{\ov{\jmath}})g(\sigma, Z_{\ov k})\right)\nabla \dif \varphi(Z_j, Z_k)^{\prime}\\
&+\left(2g(\nabla_{Z_j}\sigma, Z_k)-g(\sigma, Z_j)g(\sigma, Z_k)\right)\nabla \dif \varphi(Z_{\ov{\jmath}}, Z_{\ov k})^{\prime}\\
&+\left(2\nabla \dif \varphi (Z_j,(\nabla_{Z_{\ov{\jmath}}}\sigma)^{\prime\prime})
+ 2\nabla \dif \varphi(Z_{\ov{\jmath}}, (\nabla_{Z_j}\sigma)^{\prime})-\nabla \dif \varphi(\sigma^{\prime}, \sigma^{\prime\prime}) \right)^{\prime},
\end{align*}
since Condition \eqref{bihol} ensures that $2g(\nabla_{\cdot}\sigma, \cdot)-g(\sigma, \cdot)g(\sigma, \cdot)$ is an anti-symmetric tensor, its inner product with the second fundamental form vanishes and the formula simplifies to
\begin{align*}
[\tau_2(\varphi)]^\prime = & \, \dif \varphi \left( \tr \nabla^2 \sigma + \Ric \sigma - \nabla_{\sigma}\sigma \right)^\prime\\
&+\left(2\nabla \dif \varphi (Z_j,(\nabla_{Z_{\ov{\jmath}}}\sigma)^{\prime\prime})
+ 2\nabla \dif \varphi(Z_{\ov{\jmath}}, (\nabla_{Z_j}\sigma)^{\prime})-\nabla \dif \varphi(\sigma^{\prime}, \sigma^{\prime\prime}) \right)^{\prime}.
\end{align*}

Now, since $(N, J^N, h)$ is $(1,2)$-symplectic and $\varphi$ holomorphic, $\nabla^{\varphi}_{Z}\dif \varphi(T^{\prime\prime}M)$ is in $\varphi^{-1}T^{\prime\prime}N$ for all $Z$ in $T^{\prime}M$, so
\begin{align*}
[\tau_2(\varphi)]^\prime = & \, \dif \varphi \left( \tr \nabla^2 \sigma + \Ric \sigma - \nabla_{\sigma}\sigma \right)^\prime\\
&+ \left(- 2 \dif \varphi(\nabla_{Z_j} (\nabla_{Z_{\ov{\jmath}}}\sigma)^{\prime \prime})
- 2 \dif \varphi(\nabla_{(\nabla_{Z_j}\sigma)^{\prime}} Z_{\ov{\jmath}})+ \dif \varphi(\nabla_{\sigma^{\prime}} \sigma^{\prime \prime}) \right)^{\prime}.
\end{align*} 
Then Condition \eqref{bihol'} forces $\varphi$ to be biharmonic. 
\end{proof}

\begin{co}\label{real}
On the real tangent bundle, Conditions \eqref{bihol} and \eqref{bihol'} of Theorem \ref{theo} can be rewritten
\begin{equation}\label{biholr}
(\mathcal{L}_\sigma g)(X, Y)- (\mathcal{L}_\sigma g)(JX, JY)
=\sigma^{\flat}(X)\sigma^{\flat}(Y)-\sigma^{\flat}(JX)\sigma^{\flat}(JY),
\end{equation}
where $\mathcal{L}$ is the Lie derivative, and
\begin{equation}\label{biholr'}
\begin{split}
&\tr \nabla^2 \sigma + \Ric \sigma -\nabla_{\sigma}\sigma + \tfrac{1}{2}J\left(\left(\nabla_{\sigma} J\right)(\sigma) + \left(\nabla_{J\sigma} J\right)(J\sigma)\right)\\
&+\tfrac{1}{2}\tr \left(\left(\nabla_{\cdot} J \right) \circ \DD^{\sigma}(\cdot) -\left(\nabla_{\DD^{\sigma}(\cdot)} J\right)(\cdot)\right)=0.
\end{split}
\end{equation}
Here $\DD^{\sigma}$ is the $J$-invariant $(1,1)$-tensor field defined by $\DD^{\sigma}(X) =
\nabla_{JX}\sigma + J\nabla_{X}\sigma$.
\end{co}

\begin{proof}
Let $X$, $Y$ be vectors tangent to $M$, consider their $T^{\prime}M$-components $Z=X-\ii JX$ and $W=Y-\ii JY$ and inject them in \eqref{bihol}, then both the real and imaginary parts turn out to be equivalent to \eqref{biholr}.

To prove the second statement, first check that
\begin{equation*}
(\nabla_{Z_{\ov{\jmath}}}\sigma)^{\prime \prime}=\tfrac{1}{2\sqrt{2}}
\left(-J\DD^{\sigma}(e_j) + \ii \DD^{\sigma}(e_j)\right),
\end{equation*}
with the corresponding conjugate expression for $(\nabla_{Z_j}\sigma)^{\prime}$, then
\begin{align*}
&\nabla_{Z_j}(\nabla_{Z_{\ov{\jmath}}}\sigma)^{\prime \prime}
+\nabla_{(\nabla_{Z_j}\sigma)^{\prime}}Z_{\ov{\jmath}} =(1/4)
( \nabla_{Je_j}\DD^{\sigma}(e_j)+\nabla_{\DD^{\sigma}(e_j)}Je_j\\
&-\nabla_{e_j}\DD^{\sigma}(Je_j)-\nabla_{\DD^{\sigma}(Je_j)} e_j + \ii \left([e_j, \DD^{\sigma}(e_j)]+[Je_j, \DD^{\sigma}(Je_j)]\right)).
\end{align*}
The remaining terms in \eqref{bihol'} are easily transformed, yielding Equation \eqref{biholr'}.
\end{proof}

\section{Biharmonic holomorphic maps on l.c.K. manifolds}

In this paragraph, we exploit the conditions of Theorem \ref{theo} on locally conformally K\"ahler manifolds. This class of complex manifolds provides a quite suitable framework for our problem on at least two accounts: first by mere definition their Lee vector fields do not vanish and therefore we steer clear of the Lichnerowicz theorem and harmonic maps; second the cumbersome equations of Theorem \ref{theo} greatly simplify to take on a strong geometric flavour (Proposition \ref{lck}).

This adequacy is even more palpable in Proposition \ref{gck}, when the equations are localized on conformally K\"ahler open subsets and transformed in terms of the conformal factor. Note that, as one should expect, we encounter one the equations of \cite{BK}, since the identity map will always be holomorphic.

While biharmonicity in dimension four does not enjoy the conformal invariance of harmonic maps on surfaces, it has often been observed to possess some distinct features, as illustrated by biharmonic morphisms or the stress-energy tensor. This phenomenon occurs again here in Proposition \ref{star}, where our equations in dimension four reformulate into purely geometric conditions on the scalar and Ricci curvatures.

\begin{de} $($\cite{DO}$)$
A Hermitian manifold $(M, g, J)$ is called \emph{locally conformal K\"ahler} (l.c.K.) if each point admits a neighbourhood on which the metric is conformal to a K\"ahler metric. 
This is equivalent to the existence of a (globally defined) closed one-form $\theta$, the Lee form, such that
$$
\dif \Omega = \theta \wedge \Omega, 
$$
where $\Omega(X, Y) = g(X,JY)$ is the K\"ahler form. If $\theta$ is moreover exact, the manifold is called \emph{globally conformal K\"ahler} (g.c.K.).

\noindent If $\dim_{\RR}M > 2$, then for all vectors $X$ and $Y$
\begin{equation}\label{naj}
(\nabla_X J)(Y) =\tfrac{1}{2}\left(\theta(JY)X - \theta(Y)JX + g(X, Y)JB - \Omega(X, Y )B\right) ,
\end{equation}
where the Lee vector field $B=\tfrac{2}{2 - \dim M}J \di J$ is the dual of $\theta$.
\end{de}

\begin{pr}\label{lck}
Let $(M^m, g, J)$ be an l.c.K. manifold. If its Lee vector field $B$ and Lee form $\theta$ satisfy
\begin{equation}\label{biholck}
g(\nabla_{X}B, Y)- g(\nabla_{JX}B, JY)
=\tfrac{2-m}{4}\big(\theta(X)\theta(Y)-\theta(JX)\theta(JY)\big),
\end{equation}
for all vectors $X$, $Y$, and 
\begin{equation}\label{biholck'}
\tr \nabla^2 B + \Ric B + \tfrac{m-6}{2}\nabla_{B}B + \left(\di B - \tfrac{m-2}{4}\abs{B}^2 \right) B =0,
\end{equation}
then any holomorphic map from $(M,J,g)$ into a $(1,2)$-symplectic almost Hermitian manifold is biharmonic.
\end{pr}

\begin{proof}
Let $(M, J, g)$ be an l.c.K. manifold and $(N, J^N, h)$ a (1,2)-symplectic almost Hermitian manifold. Take $\{e_i, Je_i\}_{i=1,...,n}$ an orthonormal frame on $M$. Assume that $\varphi: (M, J) \to (N, J^N)$ is a holomorphic map, then
$$\tau(\varphi)=\tfrac{m-2}{2}\dif \varphi(B).$$
As $\theta$ is closed, Condition \eqref{biholr} immediately translates into \eqref{biholck}.

\noindent Moreover, from \eqref{naj} we infer that $(\nabla_B J)(B) = (\nabla_{JB} J)(JB) = 0$, so Equation \eqref{biholr'} becomes:
\begin{equation}\label{2lck}
\tr \nabla^2 B + \Ric B +\tfrac{m-2}{2}\nabla_{B}B + 
\left(\nabla_{e_i} J\right)(\DD^{B}(e_{i}))
-\left(\nabla_{\DD^{B}(e_{i})} J\right)(e_i)=0,
\end{equation}
since $\DD^{B}$ is $J$-invariant and $(\nabla_X J)(Y) = (\nabla_{JX} J)(JY)$ ($J$ being integrable).

\noindent Using Equation \eqref{naj} we obtain
\begin{equation*}
\begin{split}
&2\left(\left(\nabla_{e_i} J\right)(\DD^{B}(e_{i}))
-\left(\nabla_{\DD^{B}(e_{i})} J\right)(e_i)\right)\\
&=g\left(J\nabla_{Je_i}B - \nabla_{e_i}B, B\right)e_i
+g(B, e_i)\left(J\nabla_{Je_i}B - \nabla_{e_i}B \right)\\
&-g\left(\nabla_{Je_i}B + J\nabla_{e_i}B, B\right)Je_i
-g(B, Je_i)\left(\nabla_{Je_i}B + J\nabla_{e_i}B \right)\\
&-2g\left(J\nabla_{Je_i}B - \nabla_{e_i}B, e_i\right)B\\
&=g\left(J\nabla_{JB}B, e_i\right)e_i- \tfrac{1}{2}e_i(\abs{B}^2)e_i
+J\nabla_{JB}B - \nabla_{B}B\\
&- \tfrac{1}{2}(Je_i)(\abs{B}^2)Je_i + 
g\left(J\nabla_{JB}B, Je_i\right)Je_i + 2(\di B)B\\
&=2J\nabla_{JB}B - \tfrac{1}{2}\gr \abs{B}^2 - \nabla_{B}B + 2(\di B)B,
\end{split}
\end{equation*}
so that Equation \eqref{2lck} becomes
\begin{equation}\label{2lck+}
\tr \nabla^2 B + \Ric B +\tfrac{m-3}{2}\nabla_{B}B + 
\nabla_{JB}JB + (\di B)B - \tfrac{1}{4}\gr \abs{B}^2 =0.
\end{equation}
Choosing $X=B$ in \eqref{biholck} yields
$$
\nabla_{B}B + \nabla_{JB}JB = \tfrac{2-m}{4}\abs{B}^2 B,
$$ 
simplifying \eqref{2lck+} into Equation \eqref{biholck'}.
\end{proof}

\begin{pr}\label{gck}
Let $(M^m, g, J)$ be a K\"ahler manifold and $\gamma$ a smooth function on $M$ such that $J (e^{\frac{m-6}{2}\gamma}\gr\gamma)$ is a Killing vector field. If
\begin{equation}\label{biholgk}
\Delta \gamma + \tfrac{m-2}{2}\abs{\gr \gamma}^2 = 0,
\end{equation}
then any holomorphic map from $(M, e^{2\gamma}g, J)$ into a $(1,2)$-symplectic almost Hermitian manifold is biharmonic.
\end{pr}

\begin{proof}
Let $(M^m, g, J)$ be a K\"ahler manifold and $\gamma$ a smooth function on $M$. Put $\tilde g = e^{2\gamma}g$, then the Lee vector field of $(M, \tilde g, J)$ is $B = 2 e^{-2\gamma} \gr \gamma$ and Equation \eqref{biholck} becomes
$$
g\left(\nabla_{X}\gr \gamma, Y \right) - g\left(\nabla_{JX}\gr \gamma, JY \right)
=\tfrac{6-m}{2}\left(X(\gamma)Y(\gamma)-JX(\gamma)JY(\gamma)\right), 
$$
or equivalently 
\begin{equation}\label{hess}
\Hess \big(e^{\tfrac{m-6}{2}\gamma}\big)(X,Y) =
\Hess\big(e^{\tfrac{m-6}{2}\gamma}\big)(JX, JY),
\end{equation}
if $m \neq 6$ and 
\begin{equation}\label{hess'}
\Hess(\gamma)(X,Y) =
\Hess(\gamma)(JX, JY),
\end{equation}
if $m = 6$, and, on a K\"ahler manifold, the Hessian of a function $F$ is $J$-invariant if and only if $J\gr F$ is a Killing vector field. 

The individual terms of Equation \eqref{biholck'} can be expressed with respect to the metric $g$:
\begin{align*}
\tr_{\tilde g} \widetilde{\nabla}^2 B&=2e^{-4\gamma}(\tr \nabla^2 (\gr \gamma) -(3\Delta \gamma  + 2(m-2)\abs{\gr \gamma}^2)\gr \gamma \\
&+ (m-2) \nabla_{\gr \gamma}\gr \gamma );\\
\Ric_{\tilde g} B &= 2e^{-4\gamma}\left(\Ric(\gr \gamma) -(\Delta \gamma)\gr \gamma - (m-2) \nabla_{\gr \gamma}\gr \gamma \right);\\
\widetilde{\nabla}_B B &=4e^{-4\gamma}\left(-\abs{\gr \gamma}^2 \, \gr \gamma + \nabla_{\gr \gamma}\gr \gamma \right);\\
(\di_{\tilde g} B)B &=4e^{-4\gamma}\left((\Delta \gamma) \gr \gamma + (m-2)\abs{\gr \gamma}^2 \, \gr \gamma \right);\\
\abs{B}_{\tilde g}^{2} B &=8e^{-4\gamma}\abs{\gr \gamma}^2 \, \gr \gamma,
\end{align*}
so Equation \eqref{biholck'} becomes:
\begin{equation}\label{bikahler}
\begin{split}
&\tr \nabla^2 (\gr \gamma) + \Ric(\gr \gamma)
+(m-6)\nabla_{\gr \gamma}\gr \gamma \\
&-2\big(\Delta \gamma + (m-4) \abs{\gr \gamma}^2 \big) \gr \gamma = 0.
\end{split}
\end{equation}
Note that this is exactly  the condition of \cite{BK} for the identity map from $(M, \tilde g)$ to $(M, g)$ to be biharmonic.

\noindent Finally, since $J \big(e^{\frac{m-6}{2}\gamma}\gr\gamma\big)$ is a Killing vector field, it satisfies
$$
(\tr\nabla^2 +\Ric) \big(e^{\tfrac{m-6}{2}\gamma}\gr \gamma\big) =0,
$$
i.e.
\begin{equation*}
\begin{split}
&\tr \nabla^2 (\gr \gamma) + \Ric(\gr \gamma)
+(m-6)\nabla_{\gr \gamma}\gr \gamma \\
&+\tfrac{m-6}{4}\left(2\Delta \gamma + (m-6) \abs{\gr \gamma}^2 \right) \gr \gamma = 0,
\end{split}
\end{equation*}
and combined with \eqref{bikahler} this yields \eqref{biholgk}.
\end{proof}

\begin{re}\label{glck}
Since each point of an l.c.K. manifold admits a conformally K\"ahler neighbourhood, Proposition \ref{gck} extends to l.c.K. manifolds, either considering a local version of the hypothesis or rewriting them in terms of the (global) Lee form $\theta$:
\begin{equation}\label{biholck+}
\nabla\theta + \tfrac{m-2}{4}\theta \otimes \theta \ \  \text{is J-invariant},
\end{equation}
and
\begin{equation}\label{biholck+'}
\delta \theta + \tfrac{m-2}{4}\abs{\theta}^2 = 0.
\end{equation}
Notice that Condition \eqref{biholck+'} cannot be satisfied on l.c.K. manifolds with automorphic potential or on Vaisman manifolds (see \cite{orn3}).
\end{re}

\begin{de}
The $*$-Ricci tensor and the $*$-scalar curvature of a Hermitian manifold $(M, J, g)$ are defined by
$$
\Ric^*(X, Y) = \tr \big(Z \mapsto R(X, JZ)JY \big),
$$
and $s^* = \tr \Ric^*$.
\end{de}

\begin{pr}\label{star}
Let $(M, g, J)$ be a $4$-dimensional l.c.K. manifold with $J$-invariant Ricci tensor and equal scalar and $*$-scalar curvatures. Then any holomorphic map from $(M,g,J)$ into a $(1,2)$-symplectic almost Hermitian manifold is biharmonic.
\end{pr}

\begin{proof}
Let $(M^4, g, J)$ be an l.c.K. manifold with $J$-invariant Ricci tensor and $s=s^*$. Let $U$ be an open subset of $M$ on which the metric $g$ is conformal to the K\"ahler metric $g_0$, i.e. $g=e^{2\gamma}g_0$.

Then the conformal transformation law of the Ricci curvature is (\cite{Be})
$$
\Ric=\Ric_0-2(\Hess_0(\gamma) - \dif \gamma \otimes \dif \gamma)
-(\Delta_0 \gamma +2\abs{\dif \gamma}_0^2)g_0,
$$
and $\Hess_0(\gamma) - \dif \gamma \otimes \dif \gamma= -e^{\gamma}\Hess_0(e^{-\gamma})$, so if $\Ric$, $\Ric_0$ and $g_0$ are $J$-invariant, so is $\Hess_0(e^{-\gamma})$. This, in turn, is equivalent to $\Hess(e^{\gamma})$ being $J$-invariant, which is the local version of Condition \eqref{biholck+} .

Finally, the formula (\cite{Vais}) \
$\displaystyle{s - s^* = 2\delta \theta + \abs{\theta}^2}$ shows  Condition \eqref{biholck+'} to be equivalent to the equality of  $s$ and $s^*$.
\end{proof}

\section{Constructions and examples}
The main advantage of Proposition \ref{gck} is that it is tailor-made for the construction of examples (necessarily non-compact by \eqref{biholgk}). We first use it to give a recipe producing solutions to our problem from a one-dimensional foliation on a K\"ahler manifold, provided its tangent vector field is holomorphic and exact and under a condition on the eigenvalues of the pullback metric from the leaf space. Though this requirement may seem rather strong at first, it actually turns out to give a unified approach to all three examples we present.
While the first one shows that solutions can even be obtained on complex Euclidean spaces, Example \ref{cone} works for (the K\"ahler cone of) any Sasakian manifold, but in both cases a restriction on the dimension limits the solutions to complex surfaces. However, in Example \ref{inoue}, an Inoue-type construction on the product of any K\"ahler manifold and the upper half space is shown to yield examples in all dimensions.

\begin{pr}\label{subm}
Let $(M^{m}, g, J)$ be a K\"ahler manifold and $\varphi: M \to (N^{m-1}, h)$ a submersion with one-dimensional fibres spanned by a vector field $V=\gr F$, where $F$ is a positive smooth function on $M$. If $V$ is a holomorphic vector field and the eigenvalues of $\varphi^* h$ with respect to $g$ satisfy 
\begin{align*}
\prod_{i=1}^{m - 1} \lambda_i &= F^{\frac{4}{m-6}}\abs{\gr F}, \quad \text{when} \ m \neq 6;\\
\prod_{i=1}^{5} \lambda_i &= 2 F \abs{\gr F}, \quad \text{when} \ m = 6,
\end{align*}
then $\gamma=\tfrac{2}{m-6}\ln F$ (respectively $\gamma=F$), when $m \neq 6$ (respectively $m = 6$) satisfies the hypothesis of Proposition 3.
\end{pr}

\begin{lm}\label{eigenval}
Let $\varphi: (M^{m}, g) \to (N^{m-1}, h)$ be a submersion and $V$ a vertical vector field. Then
\begin{equation}\label{eigen}
\di V + V\left(\ln \frac{\prod_{i=1}^{m -1} \lambda_i}{\abs{V}}\right)=0,
\end{equation}
where $\lambda_{i}^{2}$ are the eigenvalues of $\varphi^*h$ with respect to $g$.
\end{lm}

\begin{proof}
Let $\varphi: (M^{m}, g) \to (N^{m-1}, h)$ be a submersion and $\{E_{i}, U \}_{i=1,\dots,m-1}$ an orthonormal frame of eigenvectors of $\varphi^{*}h$, i.e. $\dif \varphi (U)=0$ and  $\varphi^{*}h(E_i, X)=\lambda_{i}^{2}g(E_i, X)$ for any vector $X$. Since, for any $i$, we have (\cite{LS})
\begin{equation*}
U(\ln\lambda_{i})=g(\nabla_{E_{i}}E_{i}, U),
\end{equation*}
then
$$
\di U + U\left(\ln \prod_{i=1}^{m -1} \lambda_i\right)=0,
$$
and Equation \eqref{eigen} follows.
\end{proof}

\begin{proof}[Proof of Proposition \ref{subm}.]
Let $(M^m, g, J)$ be a K\"ahler manifold, if $V$ is gradient vector field on $M$, $V = \gr F$, then $JV$ is is a Killing vector field if and only if $V$ is (real) holomorphic. Then Equation \eqref{biholgk} from Proposition \ref{gck} translates into
\begin{equation}\label{k+}
\di V + V\left(\ln F^{\frac{4}{m-6}}\right)=0, 
\end{equation}
for $m\neq 6$, where $F = e^{\frac{m-6}{2}\gamma}$, and 
\begin{equation}\label{k++}
\di V + V\left( 2F \right)=0, 
\end{equation}
for $m = 6$, where $F = \gamma$. Conclude by combining with Equation \eqref{eigen}.
\end{proof}

\begin{ex}\label{cn}
Let $\gamma$ be a function on the complex Euclidean space $(\CC^{n}, can)$. In standard real coordinates, a Killing vector field has the general form
$$
\xi(x_1,...,x_{2n})=(-\alpha_1 x_2, \alpha_1 x_1, -\alpha_2 x_4, \alpha_2 x_3, ...),
$$
where $\{\alpha_{k}\}_{k=1,...,n}$ are real constants. Therefore $J (e^{(n-3)\gamma}\gr\gamma)$ is a Killing vector field if and only if 
\begin{align*}
\gamma(x)&=\tfrac{1}{n-3}\ln\left(\sum_{k=1}^{n} \frac{\alpha_{k}}{2}(x_{2k-1}^{2}+x_{2k}^{2})\right), \quad \text{when} \  n\neq 3,\\
\gamma(x)&=\sum_{k=1}^{3} \frac{\alpha_{k}}{2}(x_{2k-1}^{2}+x_{2k}^{2}), \quad \text{when} \ n = 3.
\end{align*}
This function is a solution of Equation \eqref{biholgk} if and only if $n=2$ and $\alpha_{1}=\alpha_{2}$, and by Proposition \ref{gck}, any holomorphic map from $\displaystyle{(\CC^{2} \setminus \{0\} , \abs{z}^{-4} can)}$ into a $(1,2)$-symplectic almost Hermitian manifold will be biharmonic.
\end{ex}

The following example can be seen as a generalization of Example \ref{cn} as $(\CC^{2}\setminus \{0\}, can)$  is isometric to the cone over the 3-sphere. 
  
\begin{ex}\label{cone}
Let $(\mathcal{C}(S), J, \widehat{g})$ be the K\"ahler cone over a Sasakian manifold $(S^{2n-1}, \phi, \xi, \eta, g)$, that is $\mathcal{C}(S)= \RR_{+}^{*} \times S$ endowed with the almost complex structure
$$
J\left(X, f \partial_r \right)=(\phi X - f\xi, \eta(X) \partial_r).
$$
and the warped product metric $\widehat{g}= \dif r^2 + r^{2}g$.

Let $\gamma=\gamma(r)$ be a smooth real function on $\mathcal{C}(S)$. In this case, Equation \eqref{biholgk} becomes 
$$
\gamma^{\prime \prime} + (2n-1)\frac{\gamma^{\prime}}{r} + (n-1)(\gamma^{\prime})^2=0
$$
with the solution $\gamma = -\ln r^2$. For this function $J (e^{(n-3)\gamma}\gr\gamma)$ is a Killing vector field if and only if $\gr r^{6-2n}$ is holomorphic which is the case only when $n=2$. Therefore, if $S$ is a Sasakian 3-dimensional manifold, then, by Proposition 3, any holomorphic map from the $(\mathcal{C}(S), r^{-4} \widehat{g})$ into a $(1,2)$-symplectic almost Hermitian manifold must be biharmonic. 
\end{ex}

\begin{ex}\label{inoue}
Let $H = \{w=w_1 + \ii w_2 \in \CC \vert w_2 >0 \}$ be the upper half plane and $(K, k, J^K)$ be a K\"ahler manifold of complex dimension $n-1$. Consider the product $H \times K$ endowed with the K\"ahler metric
$$
g = f^2(w_2) \dif w \otimes \dif \ov{w} + k,
$$
where $f:(0,\infty)\to \RR$ is a smooth function. Let $\gamma=\gamma(w_2)$ be a smooth real function on $H \times K$. For $n\neq 3$, the conditions of Proposition \ref{gck} become
$$
f F^{\prime\prime} - 2f^{\prime}F^{\prime} = 0,
$$
and
$$
(n-3)F F^{\prime\prime} + 2(F^{\prime})^2 = 0,
$$
where $F=e^{(n-3)\gamma}$. For $n = 3$, the same conditions are, respectively 
$$
f \gamma^{\prime\prime} - 2f^{\prime}\gamma^{\prime} = 0
\quad \text{and} \quad
\gamma^{\prime\prime} + 2(\gamma^{\prime})^2 = 0,
$$ 
with the solution 
$$
\gamma=\tfrac{1}{n-1}\ln w_2, \quad f(w_2)= w_{2}^{-\frac{1}{n-1}},
$$
for any $n$.

\noindent This actually produces examples in all dimensions, since any holomorphic map from the warped product $(H \times K^{n-1} , g=\dif w \otimes \dif \ov{w} + w_{2}^{\frac{2}{n-1}} k)$ into a $(1,2)$-symplectic almost Hermitian manifold will be biharmonic.
\end{ex}

\begin{re} Examples \ref{cn} and \ref{cone} can be recovered from Proposition \ref{subm} with the radial projection 
$$
(\RR_{+}^{*} \times S, \widehat{g}) \to (S, g)
$$
and Example \ref{inoue} with the projection along the second component
$$
H \times K \to \RR \times K, \quad (w_1, w_2, x) \mapsto (w_1, x).
$$
where $H \times K$ is endowed with the (K\"ahler) metric 
$$g= w_{2}^{4/(2-m)}(\dif w_{1}^{2} +\dif w_{2}^{2})+k$$ and $\RR \times K$ with the product metric $h=\dif w_{1}^{2}+k$.

In all these cases the horizontal distribution is integrable.
\end{re}

\begin{re} Let $(M,g,J)$ be a K\"ahler manifold and $\gamma$ a smooth function on $M$ such that $Id :(M, e^{2\gamma}g) \to (M,g)$ is a \emph{biharmonic morphism} (\cite{LOu}). Then any holomorphic map from $(M,e^{2\gamma}g,J)$ into a $(1,2)$-symplectic manifold is biharmonic.  

This alternative construction of biharmonic holomorphic maps can only overlap with this paper for $\dim M=4$. 
This is the case of Example \ref{cn} (which up to an isometry was given in \cite{LOu}), but not of Example \ref{inoue} since $Id :(H\times \CC , e^{2\gamma}g) \to (H\times \CC, g)$ is not a biharmonic morphism.
\end{re}

\end{document}